\documentclass[10pt]{amsart}

\usepackage{bbm}
\usepackage{algorithm}
\usepackage{algorithmic}
\usepackage{graphicx}
\usepackage{amsmath}
\usepackage{amstext}
\usepackage{amssymb}
\usepackage[english]{babel}
\usepackage{multirow}
\usepackage{tikz}
\usetikzlibrary{external}
\usetikzlibrary{positioning}
\usepackage{tabu}
\usepackage{adjustbox}

\def\T{\mathcal{T}}
\def\I{\mathcal{I}}
\def\kk{\mathbbm{k}}
\def\R{\mathfrak{R}}

\def\R{\mathcal{R}}

\DeclareMathOperator{\II}{I}
\DeclareMathOperator{\LT}{LT}

\DeclareMathOperator{\BF}{BF}
\DeclareMathOperator{\BT}{BT}

\DeclareMathOperator{\tail}{tail}
\DeclareMathOperator{\ind}{ind}

\newtheorem{theorem}{Theorem}[section]
\newtheorem{lemma}[theorem]{Lemma}

\theoremstyle{definition}
\newtheorem{definition}[theorem]{Definition}
\newtheorem{example}[theorem]{Example}
\newtheorem{remark}[theorem]{Remark}

\begin{document}
\title{Computing All Border Bases for Ideals of Points}
\author{Amir Hashemi \and Martin Kreuzer \and Samira Pourkhajouei}

\address[Amir Hashemi]{Department of Mathematical Sciences, 
Isfahan University of Technology\\ Isfahan, 84156-83111, Iran \and School of Mathematics, Institute for Research in Fundamental
Sciences (IPM), Tehran, 19395-5746, Iran}
\email{Amir.Hashemi@cc.iut.ac.ir}

\address[Martin Kreuzer]{Fakult\"at f\"ur Informatik und Mathematik
 Universit\"at Passau,
 Innstr.\ 33,
 D-94032 Passau,
 Deutschland}
\email{Martin.Kreuzer@Uni-Passau.de}

\address[Samira Pourkhajouei]{Department of Mathematical Sciences, 
Isfahan University of Technology\\ Isfahan, 84156-83111, Iran}
\email{s.pourkhajooei@math.iut.ac.ir}

\date{June 21, 2017}

\begin{abstract}
In this paper we consider the problem of computing all possible order ideals 
and also sets connected to~$1$, and the corresponding border bases, for the 
vanishing ideal of a given finite set of points. 
In this context two different approaches are discussed: 
based on the Buchberger-M\"oller Algorithm~\cite{BM}, we first propose a new 
algorithm to compute all possible order ideals and the corresponding border bases 
for an ideal of points. The second approach involves adapting the Farr-Gao 
Algorithm~\cite{Gao} for finding all sets connected to~$1$, 
as well as the corresponding border bases, for an ideal of points. 
It should be noted that our algorithms are term ordering free. 
Therefore they can compute successfully all border bases for an ideal of points. 
Both proposed algorithms have been implemented 
and their efficiency is discussed via a set of benchmarks. 

\end{abstract}

\maketitle

%
%

\section{Introduction}

The theory of {\em border bases} is a fundamental tool in computational 
commutative algebra. These bases have been developed mainly for zero-dimensional 
ideals. In this case we can consider them as a generalization of 
{\em Gr\"obner bases}, introduced by B.~Buchberger in his PhD thesis~\cite{Buch-Thesis}, 
which focuses on the structure of the quotient algebra. More precisely, 
border basis theory provides a way to find a structurally stable monomial basis for 
a zero-dimensional quotient ring of the polynomial ring, 
and it yields a special generating set for the ideal, called a border basis.
For particular choices of the monomial basis, the border basis
contains a reduced Gr\"obner basis of the ideal.

Since border bases have been shown to provide good numerical stability
(e.g., see~\cite{Ste} and~\cite{KPR}), they have been explored to 
study zero-dimensional systems with approximate coefficients obtained from empirical
measurements. Several algorithms have been designed for computing border bases, 
for instance the algorithm presented in~\cite{KK} and implemented in 
the {\sc ApCoCoA} computer algebra system (cf.~\cite{ApCoCoA}).
Border bases of zero-dimensional polynomial ideals have turned out to be 
a powerful tool in computer algebra. They have been employed to solve many 
important problems in different fields of mathematics, including linear programming, logic,  coding theory, and statistics. Many authors have worked on this topic, starting 
from the initial papers by M.G.~Marinari, M.~M\"oller and T.~Mora~\cite{zbMATH01273641}
as well as by W.~Auzinger and H.J.~Stetter~\cite{zbMATH04076472}, 
continuing with the contributions by B.~Mourrain~\cite{zbMATH01504686}, 
A.~Kehrein and M.~Kreuzer~\cite{KK}, as well as B.~Mourrain and P.~Tr\'ebuchet \cite{zbMATH06459456}, and a first textbook chapter in~\cite{book}.  
Furthermore, B.~Mourrain and P.~Tr\'ebuchet generalized 
in~\cite{zbMATH01504686,zbMATH06459456,zbMATH06420367} the notion of order 
ideals to sets connected to~$1$  which we shall call
{\em quasi order ideals} (see Section~2).
Based on this definition, they studied a generalized version of border bases, 
which we shall call {\em quasi border bases}, and  their application to
solving polynomial systems. For more details on border bases, we refer to Section~6.4 
in~\cite{book}.

Given a finite set of points, finding the ideal consisting of all polynomials vanishing 
on it, the so-called {\em vanishing ideal} of the set of points, 
has numerous applications both inside and outside of Mathematics, for example 
in statistics, optimization, computational biology, and coding theory. 
Therefore many authors have been interested in studying different aspects of 
computing vanishing ideals of finite sets of points. 
In 1982, B.~Buchberger and M.~M\"oller proposed 
in~\cite{BM} the first specialized algorithm to compute a Gr\"obner basis for 
the vanishing ideal of a set of given points. This algorithm proceeds by 
performing Gaussian elimination on a generalized Vandermonde matrix, 
and it has a polynomial time complexity. In 2006, J.B.~Farr and S.~Gao presented 
in~\cite{Gao} an incremental algorithm to compute a Gr\"obner basis for the
vanishing ideal of a set of points. However, both of these algorithms are numerically 
unstable. To address this problem, in \cite{zbMATH05500871,zbMATH05611766} the authors 
presented numerically stable algorithms to compute a border basis for an ideal of points, 
as well as its application to industrial problems. 

This leads us to the main topic of this paper, namely
to calculate all order ideals, and also all quasi order ideals, as well 
as the corresponding border bases, for an ideal of points. 
Keep in mind that all traditional algorithms to compute border bases 
rely on degree-compatible term orderings, but a zero-dimensional ideal
has border bases with respect to many order ideals which cannot
derived from a term ordering.
Let us review some previous results in this direction. 
In 2013, S.~Kaspar weakened in~\cite{kaspar} the term ordering 
requirement by introducing a {\em term marking strategy} and proposed 
an algorithm which computes border bases which cannot be obtained 
by following a term ordering strategy (see the following example). 
However, he did not provide any algorithm to find all such bases. 
Later, in~\cite{DBLP:journals/corr/BraunPX14}, G.~Braun and S.~Pokutta 
used polyhedral theory and adapted the classical border basis algorithm 
to calculate  all border bases for an ideal of points. 

Let us exhibit an example from~\cite{kaspar} which shows that there 
exists a border basis which cannot be obtained by any algorithm based 
on a term ordering strategy or the algorithm by G.~Braun and S.~Pokutta. 
Let $\mathbb{X}$ be the finite set of points $\{(1, 1),(-1, 1), (0, 0),(1, 0),(0,-1)\}$
in~$\mathbb{Q}^2$. Then the set $\{1, y, y^2, x, x^2\}$ is an order ideal 
for which the vanishing ideal of~$\mathbb{X}$ has a border basis,
namely $\{xy+x^2-1/2y^2-x-1/2y$, $x^3-x$, $x^2y-1/2y^2-1/2y,\, 
xy^2+x^2-1/2y^2-x-1/2y,\, y^3-y\}$. We note that, if we consider any term ordering, 
then the leading term of the first polynomial is either~$x^2$ or~$y^2$. 
However both terms belong to the order ideal. 
Based on the Buchberger-M\"oller Algorithm (see~\cite{BM}) 
and the Farr-Gao Algorithm (see~\cite{Gao}), we propose two different novel 
algorithms to compute, respectively, {\em all order ideals} and 
{\em all quasi order ideals}, and also the corresponding border bases, 
for an ideal of points. We have implemented both algorithms in {\sc Maple} and 
{\sc ApCoCoA} (cf.~\cite{ApCoCoA}). 
Their efficiency is discussed via several explicit examples.

The rest of the paper is organized as follows. In the next section 
we recall basic notations and definitions. 
In Sections~$3$ and~$4$ we discuss our novel approaches based on
the Buchberger-M\"oller Algorithm (resp.\ the Farr-Gao Algorithm)
to compute all order ideals (resp.\ all quasi order ideals) for which a
given ideal of points has a border basis (resp.\ a quasi border basis). 
Furthermore, we illustrate the proposed algorithms with some basic examples. 
The efficiency of the algorithms is discussed in Section~$5$ via a set of benchmarks.

%
%

\bigbreak
\section{Preliminaries}

In this section we give a brief review of basic definitions and results 
relating to Gr\"obner bases and border bases which will be used in the next sections. 
For further details we refer the reader to~\cite{book}, Section~6.4. 

Throughout this paper we let $\kk$ be a field, let $\R=\kk[x_{1},\ldots,x_{n}]$, 
and let $\T$ be the set of all terms in $ x_{1},\ldots,x_{n}$, i.e.,
\begin{center}
   $\T = \lbrace x_{1}^{\alpha_{1}}\cdots x_{n}^{\alpha_{n}} \mid 
         \alpha_{i}\geq 0,\,1 \leq i \leq n \rbrace.$
\end{center}
Here we assume that $\prec$ is a term ordering on~$\T$, i.e., a total ordering 
on~$\T$ which is multiplicative and a well-ordering. 
For a polynomial $f\in \R\setminus \{0\}$, we define its leading term, 
denoted by $\LT(f)$, to be the greatest term with respect to~$\prec$ 
which occurs in~$f$. Given an ideal $\I\subset \R$, we denote by $\LT(\I)$ 
the ideal generated by all $\LT(f)$ with $f\in \I\setminus \{0\}$. 
For a finite set $F=\{f_1,\ldots ,f_k\}\subset \R$, we write $\LT(F)$ 
for the set $\{\LT(f_1),\ldots ,\LT(f_k)\}$. A finite set $G\subset \R$ 
is called a {\em Gr\"obner basis} for $\I$ w.r.t.~$\prec$ 
if $G\subset \I$ and $\LT(\I)=\langle \LT(g) \ | \ g\in G\rangle $.

\begin{definition} Let $\mathcal{O}$ be a finite subset of~$\T$.
\begin{enumerate}
\item The set $ \mathcal{O}$ is called an {\em order ideal} \ 
if it is closed under divisors, i.e., $ t' \in \mathcal{O} $ and $ t\mid t' $ imply 
$t \in \mathcal{O}$ for all $ t,t' \in \T$. 

\item Given an order ideal $\mathcal{O} \subset \T$ 
and an ideal $\I \subset \R$, we say that~$\I$ 
{\em supports an $\mathcal{O}$-border basis}
if the residue classes of the terms in~$\mathcal{O}$ form a basis 
of~$\R/\I$ as a $\kk$-vector space.

\item If $ \mathcal{O} \subset \T$ is an order ideal, the set
$\partial \mathcal{O} = ( x_{1} \mathcal{O} \cup \cdots \cup x_{n} 
\mathcal{O}) \setminus \mathcal{O}$ is called the {\em border} of~$\mathcal{O}$. 
For the empty order ideal, we define $\partial \mathcal{O}:=\{1\}$.

\end{enumerate} 
\end{definition}
 
\begin{example}
Consider the order ideal $ \mathcal{O} =\lbrace 1,x,y,xy,x^{2},y^{2}\rbrace$ 
in~$\kk[x,y]$. Then the border of~$\mathcal{O}$ is given by
$\partial \mathcal{O} = \lbrace x^{3},x^{2}y,xy^{2},y^{3} \rbrace$. 
We illustrate~$\mathcal{O}$ and its border in the following figure.
\end{example}

\begin{figure}[h!]
 \centering
     \includegraphics[width=0.5\textwidth]{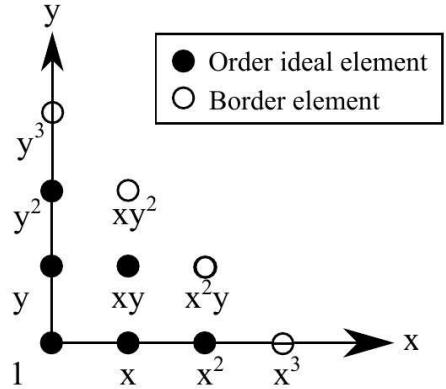}
 \caption{Depiction of an order ideal and its border}
\end{figure}

\begin{definition}\label{def:prebasis}
Let $\mathcal{O}=\lbrace t_{1}, \ldots, t_{\mu}\rbrace\subset\T$ 
be an order ideal and $\partial \mathcal{O}=\lbrace b_{1}, \ldots, b_{\nu}\rbrace $. 
\begin{enumerate}
\item A set of polynomials $G = \lbrace g_{1}, \ldots, g_{\nu}\rbrace\subset \R$ 
is called an $\mathcal{O}$-{\em border prebasis} if every~$ g_{j} $ has the form 
\begin{center}
  $ g_{j}=b_{j}-\sum_{i=1}^{\mu} \alpha_{ij}t_{i}$
\end{center}
where $ \alpha_{ij} \in \kk $.

\item If a polynomial $f \in \R \setminus \{0\} $ has the form  $f=b_{j}-\sum_{i} c_{ij}t_{i}$ with $c_{ij} \in \kk$, $t_i \in \mathcal{O}$ and $b_j \in \partial \mathcal{O}$, we say that $f$ is in $\mathcal{O}$-{\it border prebasis shape}.

\item Let $ \I\subset \R $ be a zero-dimensional ideal, 
and let $ G=\lbrace g_{1}, \ldots, g_{\nu}\rbrace$ be an $\mathcal{O}$-border 
prebasis. Then $G$ is called an $\mathcal{O}$-{\em border basis} of~$\I$  
if $G \subset \I$ and the residue classes of the elements of~$\mathcal{O}$ 
form a $\kk$-vector space basis of~$\R/\I$. 
In this case, the pair $(\mathcal{O},G)$ will be called a {\em border pair}
for~$\I$.

\end{enumerate}
\end{definition}

\begin{example}
Let $\R= \kk[x,y]$ and $\I=\langle x^{2}+2y, y^{2}-3xy+4 \rangle$. The set $\mathcal{O}=\lbrace 1,x,y,xy \rbrace $ is an order ideal and we have $\partial \mathcal{O} =\lbrace x^{2},y^{2},x^{2}y,xy^{2}\rbrace $. It is easy to check that the set $ G=\lbrace y^{2}-3xy+4,x^{2}+2y,xy^{2}+18xy+4x-24,x^{2}y+6xy-8\rbrace $ is an  $ \mathcal{O} $-border basis of $\I$.
\end{example}

In theory of border bases, we also use the concept of a border form which is 
defined as follows.

\begin{definition}
Let $\mathcal{O}=\lbrace t_{1}, \ldots, t_{\mu}\rbrace$ be an order ideal 
in~$\T$. 
\begin{enumerate}
\item For every $t'\in\T$, let $k\ge 0$ be the least number such that there
exists an index $i\in\{1,\dots,\mu\}$ and a term $t''$ of degree~$k$ such 
that $t'=t_i\, t''$. The number~$k$ is called the {\em index} of~$t'$
w.r.t.~$\mathcal{O}$ and denoted by $\ind_{\mathcal{O}}(t')$.

\item For a polynomial $f\in\R\setminus \{0\}$, we let $\ind_{\mathcal{O}}(f)$
be the largest index of a term in its support. Write $f=c_1 t'_1 + 
\cdots + c_s t'_s$ with $c_i\in\kk$ and $t'_i\in\T$. Then 
$\BF_{\mathcal{O}}(f)=  \sum_{\{i \mid \ind(t'_i)=\ind(f)\}} c_i t'_i$
is called the {\em border form} of~$f$.

\item For a polynomial $f=b_{j}-\sum_{i} c_{ij}t_{i} \in\R\setminus \{0\}$, in $\mathcal{O}$-border prebasis shape the term $b_j$ is also called the {\em border term} of~$f$ and denoted by $\BT_{\mathcal{O}}(f)$. Also, if $G$ is a set of polynomials whose elements are in $\mathcal{O}$-border prebasis shape, we denote the set $\{\BT_{\mathcal{O}}(g) \mid g \in G \}$ by $\BT_{\mathcal{O}}(G)$.
\end{enumerate}
\end{definition}

For some properties of the border form, we refer to~\cite{book}, Section 6.4.
Mourrain \cite{zbMATH01504686} introduced a generalization of order ideals, 
namely {sets connected to $1$}. Instead, for more homogeneity, we call
them {\em quasi order ideals}. They are defined as follows.

\begin{definition} Let $\mathcal{O}$ be a finite subset of~$\T$.
\begin{enumerate}
\item The {\em border} $\partial\mathcal{O}$ of $\mathcal{O}$ is defined by $\partial\mathcal{O} = (x_1\mathcal{O} \cup\cdots\cup
x_n\mathcal{O}) \setminus \mathcal{O}$. 

\item The set $\mathcal{O}$ is called a {\em quasi order ideal} 
if for every $t\in \mathcal{O}\setminus \{1\}$ we have  
$t\in \partial (\mathcal{O}\setminus \{t\})$. Also, we define $\widehat{\mathcal{O}}=\mathcal{O} \cup \partial \mathcal{O}$.

\item Given an ideal~$\I$ in~$\R$ and a quasi order ideal~$\mathcal{O}$, 
we say that~$\I$ {\em supports a quasi $\mathcal{O}$-border basis} if
the residue classes of the terms in~$\mathcal{O}$ form a basis 
of~$\R/\I$ as a $\kk$-vector space. 

\item Let $\mathcal{O}=\{t_1,\ldots ,t_\mu \}$ be a quasi order ideal 
and $ \partial \mathcal{O} = \lbrace b_{1}, \ldots, b_{\nu}\rbrace $ its border.
Then a set of polynomials $G=\lbrace g_{1}, \ldots, g_{\nu}\rbrace$ in~$\R$
is called a {\em quasi $\mathcal{O}$-border prebasis} of~$\I$ 
if $G \subset \I$ and if, for every $j\in\{1,\dots,\nu \}$, 
we have $ g_{j} = b_{j}-\sum_{i=1}^{\mu} \alpha_{ij}t_{i}$ 
where $ \alpha_{ij} \in \kk $. Also, if a polynomial $f \in \R \setminus \{0\} $ has the form  $f=b_{j}-\sum_{i} c_{ij}t_{i}$ with $c_{ij} \in \kk$, $t_i \in \mathcal{O}$ and $b_j \in \partial \mathcal{O}$, we say that $f$ is in {\it quasi} $\mathcal{O}$-{\it border prebasis shape}.

\item A quasi $\mathcal{O}$-border prebasis~$G$ is called a 
{\em quasi $\mathcal{O}$-border basis} of~$\I$ if the residue classes of the
terms in~$\mathcal{O}$ form a $\kk$-vector space basis of~$\R/\I$.
In this case the pair $(\mathcal{O},G)$ is also called a {\em quasi 
$\mathcal{O}$-border pair} for~$\I$.
\end{enumerate}
\end{definition}

\begin{example}
Let us consider the ideal $ \mathcal{I} = \langle xy+1/3y^2+x-2/3y-1,\, 
x^2-1/2y^2-x+3/2y,\, y^3-2y^2-3y \rangle$ in~$\R=\mathbb{Q}[x,y]$. Then we have 
$ \dim_{\mathbb{Q}}(\mathcal{R}/\mathcal{I})=4$. We claim that~$\I$ has
a quasi $\mathcal{O}$-border basis for the quasi order ideal 
$\mathcal{O}=\lbrace 1,x,xy,x^2y \rbrace $. To see that the residue 
classes of the terms in~$\mathcal{O}$ form a basis for~$\R/\I$, we 
let~$G$ be the reduced Gr\"obner basis of~$\I$ with respect to a term ordering~$\prec$ 
such that $y \prec x$. We consider a polynomial $f=a+bx+cxy+dx^2y$, 
where $a,b,c,d\in\mathbb{Q}$. Then the normal form of~$f$ w.r.t.~$G$ is a 
linear combination of the terms $1,x,y,y^2$, and (as long as the denominators do not 
vanish) the corresponding coefficients are $a+c+d$, $1/6(6b-6c-6d)$, 
$1/6(4c+13d)$, and $1/6(-2c-5d)$, respectively. The linear system corresponding 
to these linear polynomials has only the trivial solution. This shows 
that the residue classes of the terms in~$\mathcal{O}$ are a basis of $\R/\I$.

The set $\lbrace y,x^2,xy^2,x^3y,x^2y^2 \rbrace $ is the border of~$\mathcal{O}$.
Thus it is easy to check that the polynomials $y-2x^2y+5xy+3x-3$, 
$x^2+x^2y-xy-x$, $xy^2+xy$, $x^3y-x^2y-2xy$, and $x^2y^2+x^2y$ form a 
quasi $\mathcal{O}$-border basis of~$\mathcal{I}$.
\end{example}

To conclude this section, we briefly recall ideals of points.
For further details we refer to~\cite{book}, Section~6.3.

\begin{definition}
Let $\mathbb{X}=\lbrace P_{1},\ldots ,P_{s}\rbrace $ be a finite set 
of distinct points in $\kk^{n}$. Then the {\em vanishing ideal} of~$\mathbb{X}$ 
is defined as 
\begin{center}
   $\II(\mathbb{X}) = \lbrace f \in \R \mid f(P_{1})=\cdots =f(P_s)=0 \rbrace.$
\end{center} 
 
Furthermore, an ideal $\I$ of~$\R$ is called an {\em ideal of points} if there exists
a finite set of points~$\mathbb{X}$ in~$\kk^n$ such that $\I=\II(\mathbb{X})$.
\end{definition}

\begin{example}
Suppose that $\mathbb{X}$ contains only one point 
$P = (a_{1},\ldots,a_{n}) \in \kk^{n}$. Then we have 
$ \II(\mathbb{X})=\langle x_{1}-a_{1},\ldots,x_{n}-a_{n}\rangle$.
\end{example}

\begin{theorem}
Let $ \mathbb{X}=\lbrace P_{1},\ldots ,P_{s} \rbrace \subset \kk^{n} $
be a finite set of points.
\begin{enumerate}
\item The vanishing ideal of~$\mathbb{X}$
satisfies $\II(\mathbb{X}) = \II(P_{1}) \cap \cdots \cap \II(P_{s})$.

\item The ideal $ \II(\mathbb{X})$ is zero-dimensional, and we have
$\R/\II(\mathbb{X}) \cong \kk^{s}$.
\end{enumerate}
\end{theorem}

%
%

\section{Computing All Border Pairs}

In this section we deal with computing the set of all order ideals 
associated to the ideal of points of a given finite set of points. 
Our approach in this section relies on the Buchberger-M\"oller Algorithm~\cite{BM} 
which is an efficient algorithm to compute a Gr\"obner basis for an ideal of points. 
Before we sketch our algorithm, we first recall the classical version of the 
Buchberger-M\"oller Algorithm from \cite[p.~392]{book}. 
It takes as input a finite set of points $\mathbb{X} $ and a term 
ordering $\prec$ and returns the reduced Gr\"obner basis of $\II(\mathbb{X})$. 
Further, a variant of this algorithm outputs a set of terms $\mathcal{O}$ 
so that the residue classes of its elements form  a basis for $\R/\II(\mathbb{X})$ 
as a $\kk$-vector space. In the following we describe a presentation of  
this algorithm in which we use the {\sc DivisionAlgorithm} which receives 
as input a linear polynomial $f$, a set $G=\{g_1,\ldots ,g_m\}$ of linear 
polynomials in $y_1,\ldots ,y_s$ and a term ordering $\prec$ 
with $y_m\prec \cdots \prec y_1$  and returns a pair $p=(r, Q)$ 
where $r$ is normal remainder of $f$ with respect to~$G$ 
and a tuple $Q=[q_1,\ldots ,q_m]$ such that $f=q_1g_1+\cdots +q_mg_m+r$. 
Moreover, the function {\sc NormalForm} computes the normal remainder 
of the {\sc DivisionAlgorithm}. For further details, we refer 
to~\cite{book1}, Section~1.6.

\begin{algorithm}[H]
\caption{{\sc Buchberger-M\"oller}}
\begin{algorithmic}[1]
  \STATE {\bf{Input:}} $ \mathbb{X}= \lbrace P_{1}, \ldots,P_{s} \rbrace \subset \kk^{n}$ and a term ordering $\prec$
 \STATE {\bf{Output:}} The reduced Gr\"obner basis $G$ of $I(\mathbb{X})$  w.r.t $\prec$
 \STATE $G:=\{\}; \mathcal{O}=\{\};M:=\{\};S:=\{\}; L:=\lbrace 1 \rbrace; $
 \WHILE { $L \neq \emptyset$}
\STATE Select and remove $t:=\min_{\prec}(L)$ from $L$
\STATE Let $ P:=(r,Q)=$ {\sc DivisionAlgorithm}$(\sum_{i=1}^{s}t(P_{i})y_{i},M,\prec)$
\IF {$ r= 0  $}
\STATE $ G:=G \cup \{t- \sum _{i=1}^{m} q_{i}s_{i}\}$ where $ S=[s_{1},\ldots,s_{m} ]$ 
\STATE Remove from $L$ the terms which are multiples of $t$
\ELSE
\STATE Add $r$ to $M$
\STATE Add $t- \sum _{i=1}^{m} q_{i}s_{i}$ to $S$ where $ S=[s_{1},\ldots,s_{m}]$
\STATE $\mathcal{O}:=\mathcal{O} \cup \{t\}$
\STATE Add to $L$ those elements of $\lbrace x_{1}t,\ldots,x_{n}t \rbrace $ which are not multiples of an element in $\LT(G)\cup L$
\ENDIF
\ENDWHILE
\STATE {\bf{return}}$(G)$
\end{algorithmic}
\end{algorithm}

Below we discuss some details of this algorithm which are useful to prove 
its termination and correctness (see \cite[page 392]{book}). 
In 2011, Kreuzer and Poulisse \cite{subideal} introduced a variant 
of the Buchberger-M\"oller Algorithm to compute a border basis for an ideal of points. 
In the following, we present a variant of this algorithm for computing 
a border pair for an ideal of points. 

\begin{algorithm}[H]
\caption{{\sc BM-border}}
\begin{algorithmic}[1]
\STATE {\bf{Input:}} $ \mathbb{X}= \{P_1, \dots, P_s \} \subseteq \kk^{n}$ 
      and a term ordering $\prec$

\STATE {\bf{Output:}} A border pair $(\mathcal{O},G)$ for $ \II(X) $
\STATE $G:=\{\}; \mathcal{O}=\{\};M:=[\ ];S:=[\ ]; L:=\lbrace 1 \rbrace; $
  \WHILE {L $\neq \emptyset$}
    \STATE Select and remove from $L$ an element $t$ of minimal degree
    \STATE Let $ P:=(r,Q)=$ {\sc DivisionAlgorithm}$(\sum_{i=1}^{s}t(P_{i})y_{i},M,\prec)$
    \IF {$ r= 0  $}
      \STATE $ G:=G \cup \{t- \sum _{i=1}^{m} q_{i}s_{i}\}$ where $ S=[s_{1},\ldots,s_{m} ]$
    \ELSE
    \STATE Add $r$ to $M$ and $t$ to $\mathcal{O}$
    \STATE Add $t- \sum _{i=1}^{m} q_{i}s_{i}$ to $S$ where $ S=[s_{1},\ldots,s_{m}]$ 
    \STATE $L:=L \cup  \lbrace x_{1}t,\ldots,x_{n}t \rbrace $ 
    \ENDIF
  \ENDWHILE
\STATE {\bf{return}}$(\mathcal{O},G)$
\end{algorithmic}
\end{algorithm}

The following two lemmata are used to  prove the termination and 
correctness of this algorithm.

\begin{lemma}\label{lem1}
With the notations of this algorithm, the set $M$ is a set of linear polynomials 
in $y_1,\ldots ,y_s$ which is a Gr\"obner basis of the ideal generated by 
$\sum_{i=1}^{s}t(P_{i})y_{i}$ for $t\in \mathcal{O}$ according to $\prec$.
\end{lemma}

\begin{proof}
We argue by induction on the size $m$ of $M$. By Steps~6 and~10 of the algorithm,
we first add $y_1+\cdots +y_s$ (corresponding to $1$) to~$M$. So, the assertions 
hold when $m=1$. Now, suppose that $M$ is a Gr\"obner basis containing linear polynomials, 
and we consider a linear polynomial $r$. The polynomial $r$ is the normal form 
of a polynomial w.r.t. $M$ when we add $r$ to $M$. 
Using the Buchberger criterion (cf.~\cite[Section 2.5]{book1}), 
since all the polynomials are linear and their leading terms are pairwise coprime, 
it is straightforward to check that the result of adding $r$ to $M$ is 
indeed a Gr\"obner basis. 
\end{proof}

\begin{lemma}\label{lem2}
Suppose that a linear combination of a term $t$ and the elements in $\mathcal{O}$ 
belongs to $\II(\mathbb{X})$. Then we have 
{\sc NormalForm}$(\sum_{i=1}^{s}t(P_{i})y_{i},M,\prec)=0$ and vice versa.
\end{lemma}

\begin{proof}
Suppose that a linear combination of $t$ and the elements in $\mathcal{O}$ 
belongs to $\II(\mathbb{X})$. It follows that $\sum_{i=1}^{s}t(P_{i})y_{i}$ 
is a linear combination of the elements of the set 
$F=\{\sum_{i=1}^{s}u(P_{i})y_{i} \ | \ u\in \mathcal{O}\}$. 
On the other hand, by Lemma \ref{lem1}, the set $M$ is a Gr\"obner basis 
of the ideal generated by the set $F$ and therefore one has 
{\sc NormalForm}$(\sum_{i=1}^{s}t(P_{i})y_{i},M,\prec)=0$. The converse is obvious.
\end{proof}

\begin{theorem}
Given a finite set of points $\mathbb{X}$, algorithm {\sc BM-Border} 
terminates and returns a border pair for $\II(\mathbb{X})$.
\end{theorem}

\begin{proof}
First we show that the algorithm terminates. Reasoning by reductio ad absurdum, 
we assume that the algorithm does not terminate. Thus, by Steps~10 and~12 in the 
algorithm, it follows that $\mathcal{O}$ is infinite, since $L$ is enlarged only when $\mathcal{O}$ is enlarged. We observe that no linear combination of the terms in $\mathcal{O}$ belongs to $\II(\mathbb{X})$ (Lemma \ref{lem2}). This entails that $\mathcal{O}$ can be extended to a basis for $\R/\II(\mathbb{X})$ as a $\kk$-vector space. This contradicts the zero-dimensionality of $\II(\mathbb{X})$, and so the algorithm terminates.

Now we claim that $\mathcal{O}$ is an order ideal. Suppose that $t\in \mathcal{O}$ and $\tilde{t}=t/x_i \notin \mathcal{O}$ for some $i$. Since $\tilde{t} \notin \mathcal{O}$, the normal form of $\tilde{t}$ is a linear combination of the normal forms of the elements in $\mathcal{O}$ computed before $\tilde{t}$. If we multiply both sides of this representation by $x_i$, then we obtain a linear combination of elements of $\mathcal{O}$ for $t$  which is a contradiction to  $t\in \mathcal{O}$. 
  
  We conclude the proof by showing that $G$ is a border basis w.r.t. $\mathcal{O}$. The set $L$ is enlarged only in Step 12. In the set $L$, for each $t\in \mathcal{O}$ and for each $x_i$ we consider $x_it$.  If a linear combination of $x_it$ and the elements in $\mathcal{O}$ belongs to $\II(\mathbb{X})$, then by Lemma \ref{lem2} we have   {\sc NormalForm}$(\sum_{i=1}^{s}t(P_{i})y_{i},M,\prec)=0$ and so we add the polynomial $t- \sum _{i=1}^{m} q_{i}s_{i}$ to $G$ which finally shows the set $G$ has the form a prebasis. On the other hand, in each iteration of Step 10, we add a term $t$ to $\mathcal{O}$ which is a linearly independent from the remainders of the terms in $\mathcal{O}$. Since the set $\mathcal{O}$ has $s$ terms, then the set $\mathcal{O}$ forms a basis for the $\kk$-vector space $\R/\II(\mathbb{X})$ which shows that $G$ is a border basis and the proof is finished.
\end{proof}

\begin{remark}
In this algorithm, the list $L$ is considered to be a set, 
and so repeated terms are removed. Further, due to the degree-compatible  
selection strategy of this algorithm, it does not reconsider a term to study. 
Finally, we note that using this algorithm, one can obtain 
a border basis for an ideal of points so that the border terms of its elements 
do not respect any term ordering. For example, if we consider any term ordering, 
we can not obtain the set $\{1, y, x, y^2, x^2\}$  as the complement of a 
leading term ideal for the vanishing ideal of the set of points  
$\mathbb{X}=\{(1, 1),(-1, 1), (0, 0),(1, 0),(0,-1)\}$.  
However, all previous algorithms rely on a term ordering. 
\end{remark}

In the following example we use algorithm {\sc BM-Border} 
to obtain the border basis mentioned in the introduction of~\cite{kaspar}.

\begin{example}
Let us execute the steps of algorithm {\sc BM-Border} to compute an order ideal 
and a border basis for the ideal of points of the set 
$\mathbb{X}= \{ (0, 0),(0,-1),\allowbreak (1, 0), (1, 1), (-1, 1) \}$ in $\mathbb{Q}^2$. 
We number the iterations of the \textbf{while}-loop in this algorithm consecutively.

\begin{itemize}
\item[] First we set $G:=\{\}; \mathcal{O}=\{\};M:=[\ ];S:=[\ ]; L:=\lbrace 1 \rbrace $.

\item[$(1)$] We select $ t=1$. We have $f=y_{1}+y_{2}+y_{3}+y_{4}+y_{5}$. 
Thus, $  M=[y_{1}+y_{2}+y_{3}+y_{4}+y_{5}], \, \, \mathcal{O}=\{1\}, \, \, 
S=[1],\, \, \text{and} \, \,L=\{x,y\} $.

\item[$(2)$] Choose $ t=x$ and get  $L=\{y\}$. We have $f=y_{3}+y_{4}-y_{5}$. Thus, $  M=[y_{1}+y_{2}+y_{3}+y_{4}+y_{5},y_{3}+y_{4}-y_{5}], \, \,\mathcal{O}=\{1,x\}, \, \, S=[1,x], \, \, \text{and} \, \,L=\{y,x^{2},xy\} $.
     
\item[$(3)$] Choose $ t=y$ and get  $L=\{x^{2},xy\}$. We have $f=-y_{2}+y_{4}+y_{5}$. Thus, $ M=[y_{1}+y_{2}+y_{3}+y_{4}+y_{5},y_{3}+y_{4}-y_{5},-y_{2}+y_{4}+y_{5}], \, \,\mathcal{O}=\{1,x,y \}, \, \, S=[1,x,y], \, \, \text{and} \, \,L=\{x^{2},y^{2},xy\} $.
   
  \item[$(4)$] Choose $ t=x^{2}$ and get $L=\{y^{2},xy\}$. We have $f=2y_{5}$. Thus, $ M=[y_{1}+y_{2}+y_{3}+y_{4}+y_{5},y_{3}+y_{4}-y_{5},-y_{2}+y_{4}+y_{5},2y_{5}], \, \,\mathcal{O}=\{1,x,y,x^{2} \}, \, \, S=[1,x,y,x^{2}-x], \, \, \text{and} \, \,L=\{y^{2},xy,x^3,x^{2}y\} $. 
   
\item[$(5)$] Choose $ t=x^{3}$ and get $L=\{y^{2},xy,x^{2}y\}$. Since $f=0$ and $ g_{1}=x^{3}-x$ we have $ G=\{x^{3}-x\}$.

 \item[$(6)$] Choose $ t=y^{2}$ and get $L=\{xy,x^{2}y\}$. We have $f=2y_{4}$. Thus, $ M=[y_{1}+y_{2}+y_{3}+y_{4}+y_{5},y_{3}+y_{4}-y_{5},-y_{2}+y_{4}+y_{5},2y_{5},2y_{4}], \, \,\mathcal{O}=\{1,x,y,x^{2},y^{2} \}, \, \,S=[1,x,y,x^{2}-x,-x^2+y^{2}+x+y], \, \,  \text{and} \, \,L=\{y^{3},xy,x^{2}y,xy^2\} $.

\item[$(7)$] Choose $ t=y^{3}$ and get $L=\{xy,x^{2}y,xy^2\}$. Since $f=0$, we compute $ g_{2}=y^{3}-y$. Now we have $ G=\{x^{3}-x,y^{3}-y\}$.

\item[$(8)$] Choose $ t=xy$ and get $L=\{x^{2}y,xy^2\}$. Since $f=0$, we compute $ g_{3}=xy-x+x^2-1/2y-1/2y^2$. Now we have $ G=\{x^{3}-x,y^{3}-y,xy-x+x^2-1/2y-1/2y^2\}$.
       
 \item[$(9)$] Choose $ t=x^2y$ and get $L=\{xy^2\}$. Since $f=0$, we compute $ g_{4}=x^2y-1/2y-1/2y^2$. Now we have $ G=\{x^{3}-x,y^{3}-y,xy-x+x^2-1/2y-1/2y^2,x^2y-1/2y-1/2y^2\}$.      
       
\item[$(10)$] Finally we select $ t=xy^2$ and compute the polynomial $ g_{5}=y^2x-x+x^2-1/2y-1/2y^2  $. Since $L=\{\}$, we obtain $\mathcal{O}=\{1,x,y,x^2,y^2\} $ and $G=\{x^{3}-x,y^{3}-y,xy-x+x^2-1/2y-1/2y^2,x^2y-1/2y-1/2y^2,y^2x-x+x^2-1/2y-1/2y^2\}$.   
 
\end{itemize}
\end{example}

Based on the above algorithm, we propose a new recursive algorithm 
to compute all order ideals for which a given ideal of points supports a border basis.

\begin{algorithm}[ht]
\caption{{\sc BM-AllOrderIdeals}}
\label{Allordermain}
\begin{algorithmic}[1]
\STATE {\bf{Input:}} $ \mathbb{X}= \lbrace P_{1},\ldots,P_{s} \rbrace \subset \kk^{n}$

\STATE {\bf{Output:}} A list of all order ideals $\mathcal{O}$ such that
$I(\mathbb{X}) $ has an $\mathcal{O}$-border basis

\STATE $ L:=\emptyset, \mathcal{O}:=\emptyset$

\STATE Let $  M:=\text{Mat}_{0,s}(\kk) $ be a matrix with $s$ columns and zero rows

\STATE{\sc AllOIStep}$(\mathbb{X},L,\mathcal{O},M)$

 \STATE{\bf{return}}  $(L)$
\end{algorithmic}
\end{algorithm}

Here subalgorithm {\sc AllOIStep}$(\cdots)$ is given by Algorithm~4. 
Notice that, for a term $t$, we let $\text{eval}_{\mathbb{X}}(t)$ 
be the evaluation vector $\text{eval}_{\mathbb{X}}(t)=(t(P_1),\ldots,t(P_s))$
with respect to the given set of points $\mathbb{X}= \{P_1, \dots, P_s \} 
\subset \kk^{n}$.

\begin{algorithm}[ht]
\caption{{\sc  AllOIStep }}
\label{Allorder}
\begin{algorithmic}[1]

\STATE {\bf{Input:}} $ \mathbb{X}= \{ P_1, \dots, P_s \} \subset \kk^{n}$, 
    a list of sets $L$, an order ideal $\mathcal{O}$, and a matrix $M$

\STATE {\bf{Output:}} An updated tuple $(\mathbb{X},L,\mathcal{O},M)$

\IF {$|\mathcal{O}|= s $}
  \STATE Append the order ideal $\mathcal{O}$ to $L$
\ENDIF

\IF {$|\mathcal{O}|< s $}
  \STATE Let $S$ be the set of all terms $t \not \in \mathcal{O}$ 
         s.t $ \mathcal{O} \cup \{t\}$ is an order ideal 
  \FOR {$t$ \textbf{in} $S$}
    \STATE Compute the reduction $(v_1,\ldots,v_s)$ of $(t(P_1),\ldots,t(P_s))$ 
           with respect to $M$ and write
   \begin{center}
   $(v_1,\ldots,v_s)=(t(P_1),\ldots,t(P_s))-\sum_{k}c_k(m_{k1},\ldots,m_{ks})$ 
   \end{center}
   where $c_k \in \kk$ and $(m_{k1},\ldots,m_{ks})$ are the rows of $M$
   \IF {$(v_1,\ldots,v_s)\neq (0,\ldots,0)$}
   \STATE Let $M_{new}$ be the matrix obtained by appending $(v_1,\ldots,v_s)$ 
          as a new row to $M$
   \STATE $\mathcal{O}:=\mathcal{O} \cup \{t\}$
   \STATE {\sc AllOIStep}$(\mathbb{X},L,\mathcal{O},M_{new})$
   \ENDIF
   \ENDFOR
\ENDIF

\end{algorithmic}
\end{algorithm}

In Section~5 we analyze the performance of this algorithm. 
Before proving its correctness, let us apply it to the set 
$\mathbb{X}=\{(0, 0), (0, -1),(1, 0),(1, 1), (-1, 1)\}$ and explain the main idea. 
It should be noted that a more detailed application of the algorithm is given in 
Example~\ref{ex}. As we can see in Figure~2, we select successively the terms 
$1,x,y,x^2$ and add them to $\mathcal{O}$. Then the first border of $\mathcal{O}$ 
is $\{x^3,x^2y,y^2,xy\}$. We need to study each of these terms, except $x^2y$ 
because $\{1,x,y,x^2,x^2y\}$ does not form an order ideal. If we choose $x^3$ 
then we find a linear dependency and the corresponding branch is broken. 
However, if we choose $y^2$, its evaluation vector is linearly independent 
from the rows of $M$ and therefore we add $y^2$ to  $\mathcal{O}$. 
Since the cardinality of the resulting set is~5, we found an 
order ideal $\mathcal{O}$ as desired.

\begin{center}
\begin{figure}
\begin{tikzpicture}[sibling distance=16mm,rounded corners]
  \node [rectangle,draw] (0) {1}
  child {node [rectangle,draw] (1) {$x$}
    child {node [rectangle,draw] (3) {$y$}
      child {node [rectangle,draw] (q7) {$x^{2}$} 
       child {node [rectangle,draw] (q2) {$x^{3}$ } 
        child[grow=left] {node [rectangle,draw] (2) {Break} } }
         child {node [rectangle,draw] (q2) {$y^{2}$} 
        child [sibling distance=5cm]  {node [rectangle,draw]  (2)
        [text width=4cm,text height=-0.3cm, sibling distance=10cm] 
        {$$\mathcal{O}=\lbrace1,x,y,x^{2},y^{2}\rbrace$$} 
        child {node [rectangle,draw]  [text width=12cm,text height=0.3cm,] 
         {{\footnotesize $ G=\lbrace x^2y-1/2y^2-1/2y, x^3-x, y^3-y,$
          $y^2x+x^2-x-1/2y^2-1/2y, xy+x^2-x-1/2y^2-1/2y \rbrace $}} } }}
          child {node [rectangle,draw] (e) {$xy$} child[grow=down] {node [below right=of e] (q10) {} edge from parent[dashed]} } }
      child {node [rectangle,draw] (a) {$y^{2}$} 
      child {node[below right=of a]  (q10) {} edge from parent[dashed]}}
        child {node [rectangle,draw] (b) {$xy$} child[grow=right] {node[below right=of b] (q10) {} edge from parent[dashed]}}}
    child {node [rectangle,draw] (c) {$x^{2}$} child[grow=right]  {node[below right=of c] (q10) {} edge from parent[dashed]} }}  
  child {node [rectangle,draw] (d) {$y$} 
   child[grow=right] {node[below right=of d] (q10) {} edge from parent[dashed]}};
\end{tikzpicture}
\caption{Example for the recursive structure of Algorithm~3}
\end{figure}
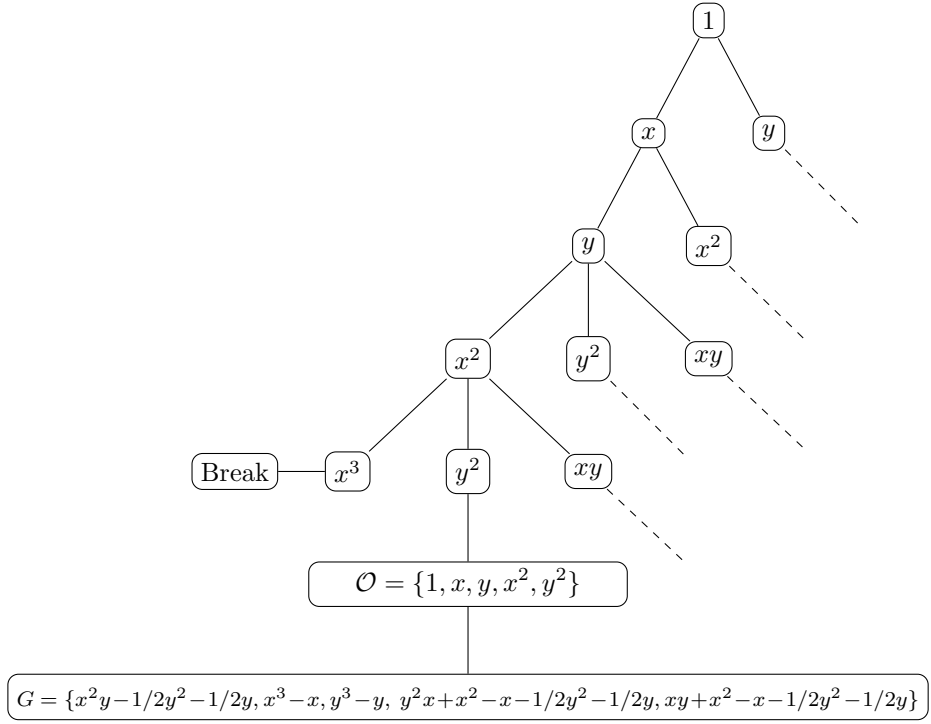
\end{center}

\begin{theorem}
Algorithm {\sc BM-AllOrderIdeals} terminates and computes all order ideals $\mathcal{O}$ such that the vanishing ideal of the given set of points has an $\mathcal{O}$-border basis.
\end{theorem}

\begin{proof}
First we discuss the termination of the algorithm. By Step~7 of 
Algorithm~4, at each recursion we consider a new term~$t$ 
and a set $\mathcal{O}$  which is an order ideal. 
In Step 7, Algorithm~4 considers $\mathcal{O}=\mathcal{O}\cup \{t\}$ 
as the new order ideal and creates new branches (Step 8 in the {\bf for}-loop) 
for each term~$t$ in the border of~$\mathcal{O}$ so that $\mathcal{O}\cup \{t\}$ 
forms an order ideal. Hence, for~$t$ we have a finite number of choices. 
So, the order ideals considered in the next level of the recursion will have 
one more element and eventually we reach the case $|\mathcal{O}|= s$ in which 
the branch of the recursion stops. Therefore, the termination of the algorithm 
follows from the facts that each branch has length at most $s$ and each node 
has a finite number of choices.

Now we show correctness. We prove that if $\mathcal{O} = \{ t_1, \dots, t_s \}$ 
is an element in~$L$, then it is an order ideal for $\II(\mathbb{X})$. 
By Step~7 of Algorithm~4, we see that $\mathcal{O}$ is an order ideal. 
It remains to prove that $\II(\mathbb{X})$ has an $\mathcal{O}$-border basis. 
Since the set $\mathcal{O}$ has $s$ terms,  it has the correct cardinality 
for $\II(\mathbb{X})$ to support an $\mathcal{O}$-border basis.  
In the Step 10 of Algorithm~4, if the evaluation vector 
$(t(P_1),\ldots,t(P_s))$ of an element $t\in \mathcal{O}$ is linearly 
independent of the rows of $M$, we add it to the intermediate matrix~$M$. 
Therefore the final matrix $M$ is a square matrix whose rows correspond 
to the evaluation vectors $(t_{i}(P_1),\ldots,t_{i}(P_s))$ for each $i=1,\ldots ,s$.
    
Since $M$ is invertible, the residue classes of the terms in $\mathcal{O}$ form a 
basis for the $\kk$-vector space $\R/\II(\mathbb{X})$ by~\cite[Sec. 6]{book}. 
By the definition of border bases, $\II(\mathbb{X})$ has an $\mathcal{O}$-border 
basis which proves the claim.
  
Finally, we show that we can find any order ideal $\mathcal{O}$ of~$\II(\mathbb{X})$ 
in~$L$. For this, suppose that  $ \mathcal{O}=\lbrace t_{1},\ldots,t_{s} \rbrace $ 
is an order ideal of $\II(\mathbb{X})$  and that it is ordered increasingly according 
to the degree of its elements. Let $d=\max\{\deg(m)\ | \ m\in \mathcal{O}\}$. 
For each $i=0,\ldots ,d$, let $\mathcal{O}_i$ be the set of all terms in $\mathcal{O}$ of degree at most $i$. We prove, using induction on $i$, that every $\mathcal{O}_i$ is constructed during the algorithm. It is clear that $\mathcal{O}_0=\{1\}$ is considered by the algorithm. Since $ \mathcal{O}$ is closed under forming divisors, the set $\mathcal{O}_i$ is an order ideal as well. Now suppose that $\mathcal{O}_i=\{t_1,\ldots ,t_j\}$ has been already constructed and $t_{j+1},\ldots ,t_\ell$ is the sequence of all terms in $\mathcal{O}$ of degree $i+1$. Our goal is to prove that $A=\mathcal{O}_i\cup \{t_{j+1},\ldots ,t_{l}\}$ is used as input for {\sc AllOIStep}($\cdots$) at some point during the recursion. We proceed by induction on $k$ and show that $A_{k}=\mathcal{O}_{i} \cup \{t_{j+1},\ldots, t_{j+k}\}$ will be chosen. For the case $k=0$, we have $A_{0}=\mathcal{O}_{i}$. Since $t_{j+1} \in \partial \mathcal{O}_i \cap \mathcal{O}$, the tuple $\text{eval}_{\mathbb{X}}(t_{j+1})$ is $\kk$-linearly independent of the previous rows of $M$. So we can add it to $\mathcal{O}_{i}$, and therefore $A_1=\mathcal{O}_i \cup \{t_{j+1}\}$ will be constructed. Now suppose that $A_{k-1}$ has been constructed. We can repeat the same argument as in the case $k=1$. Namely, $t_{j+k}$ is in the border of $\mathcal{O}_i \cup \{t_{j+1},\ldots,t_{j+k-1} \}$ and $\text{eval}_{\mathbb{X}}(t_{j+k})$ is linearly independent of the rows of $M$ because of $t_{j+k} \in \mathcal{O}$. Hence $A_k=\mathcal{O}_i \cup \{t_{j+1},\ldots,t_{j+k} \}$ and $\mathcal{O}_{i+1}$ will be constructed by the algorithm. Finally, when we reach $i=d$, we have $\mathcal{O}_{i}=\mathcal{O}$ and $\mathcal{O}$ is appended to $L$ in Step~4 of {\sc AllOIStep}($\cdots$).
\end{proof}

The next example illustrates this procedure.

\begin{example}\label{ex}
Let us compute all order ideals of the ideal of points of the set 
$\mathbb{X} =  \{ (2,3), (1,4),(5,0) \}$ in $\mathbb{Q}^2$. 
In what follows we write down the steps of the above algorithm. 

\begin{enumerate}
\item[(3)] Let $ L=\{ \},\mathcal{O}=\{ \}$
 
\item[(4)] Let $ M=
 \begin{pmatrix}
   &  & 
 \end{pmatrix} $ in $\text{Mat}_{0,3}(\kk)$

\item[(5)] Call {\sc AllOIStep}($\cdots$)

\item[{[7]}] Let $S=\{ 1\}$

\item[{[9]}] Choose $ t=1$ and compute $\text{eval}_{\mathbb{X}}(t)=(1,1,1)=(v_1,v_2,v_3) $

\item[{[11]}] $M_{new}=\begin{pmatrix}
  1 & 1 & 1
 \end{pmatrix}$
 
\item[{[12]}] Let $\mathcal{O}=\{1\}$
 
\item[{[13]}] Call {\sc AllOIStep}$(\cdots)$
 
\item[{[7]}] $S=\{x,y\}$

\item[{[9]}] Choose $ t=x$ and let $\text{eval}_{\mathbb{X}}(t)=(2,1,5)$ Compute $(v_1,v_2,v_3)=(2,1,5)-2(1,1,1)=(0,-1,3)$

\item[{[11]}] $M_{new}=\begin{pmatrix}
  1 & 1 & 1\\
  0 &-1 & 3
 \end{pmatrix}$
 
\item[{[12]}] Let $ \mathcal{O}=\{1,x \}$

\item[{[13]}] Call {\sc AllOIStep}$(\cdots)$

\item[{[7]}] $ S=\{y,x^2\} $

\item[{[9]}] Choose $ t=y$ and let $\text{eval}_{\mathbb{X}}(t)=(3,4,0)$. Compute $(v_1,v_2,v_3)=(0,0,0) $
  
\item[{[9]}] Choose $ t=x^2$ and let $\text{eval}_{\mathbb{X}}(t)=(4,1,25)$.  Compute $(v_1,v_2,v_3)=(0,0,12) $
  
\item[{[11]}]  $M_{new}=\begin{pmatrix}
  1 & 1 & 1\\
  0 &-1 & 3\\
  0& 0 & 12
 \end{pmatrix} $ 

\item[{[12]}] Let $\mathcal{O}=\{1,x,x^2\}$

\item[{[4]}] Since $|\mathcal{O}|=3$, we set $L=[ \{1,x,x^2 \} ]$
 
\item[{[9]}] Choose $ t=y$ and let $\text{eval}_{\mathbb{X}}(t)=(3,4,0)$. Compute $(v_1,v_2,v_3)=(0,1,-3) $

\item[{[11]}] $M_{new}=\begin{pmatrix}
  1 & 1 & 1\\
  0 &1 & -3
 \end{pmatrix}$
 
\item[{[12]}] Let $ \mathcal{O}=\{1,y \}$

\item[{[13]}] Call {\sc AllOIStep}$(\cdots)$

\item[{[7]}] $ S=\{x,y^2\} $

\item[{[9]}] Choose $ t=x$ and let $\text{eval}_{\mathbb{X}}(t)=(2,1,5)$.  
Compute $(v_1,v_2,v_3)=(0,0,0) $
  
\item[{[9]}] Choose $ t=y^2$ and let $\text{eval}_{\mathbb{X}}(t)=(9,16,0)$.  
Compute $(v_1,v_2,v_3)=(0,0,12) $
  
\item[{[11]}]  $M_{new}=\begin{pmatrix}
  1 & 1 & 1\\
  0 &-1 & 3\\
  0& 0 & 12
 \end{pmatrix} $ 

\item[{[12]}] Let $\mathcal{O}=\{1,y,y^2\}$

\item[{[4]}] Since $|\mathcal{O}|=3$, we add $\{1,y,y^2\}$ to $L$

\item[(6)] $L=[ \{1,x,x^2 \},\{1,y,y^2\} ]$
\end{enumerate}
  
Thus $\II(\mathbb{X})$ has border bases with respect to the two order ideals $\{1,x,x^{2}\} $ and $\{1,y,y^{2}\}$. 
\end{example}
One drawback of this algorithm is that it may produce the same order ideal 
several times, as one can see in the following figure. 
However, keep in mind that our aim is to calculate all order ideals of the 
vanishing ideal of the given set of points. Thus we are willing to pay the cost 
of computing repeated results and remove them later.

\begin{figure}[H]
\begin{center}
\begin{tikzpicture}[sibling distance=25mm]

  \node  (d) {1}
  child {node  {$x$}
    child {node  {$x^2$}
      child {node  {$y$} child {node (q10) {} edge from parent[dashed]} } }
    child {node  (c) {$y$} 
       child {node (a)  {$x^{2}$}  child {node (q10) {} edge from parent[dashed]} 
    } child {node  {} edge from parent[dashed]}}}  
   child {node[below right=of d] (q10) {} edge from parent[dashed]};
\end{tikzpicture}
\end{center}
\end{figure}

Finally, we draw the attention of the reader to the example 
$\mathbb{X}= \{ (0,0,0,1),\allowbreak (1,0,0,2), (3,0,0,2),(5,0,0,3),(-1,0,0,4),(4,4,4,5),(0,0,7,6) \}$ from~\cite{DBLP:journals/corr/BraunPX14}. Algorithm~4 computes $55$  different order ideals for $\II(\mathbb{X})$. However, using the algorithm of Braun 
and Pokutta (cf. \cite{DBLP:journals/corr/BraunPX14}), one can find only~45 order ideals.

%
%

\section{Computing All Quasi Border Pairs}

Farr and Gao in \cite{Gao}  described an alternate method, which is a generalization of Newton's interpolation for univariate polynomials, to compute the reduced Gr\"obner basis for an ideal of points. Based on this incremental algorithm, we describe a new algorithm  to calculate the set of all quasi border pairs associated to an ideal of points in this section. Furthermore, we show a detailed example of the execution of this algorithm. 

In \cite[Sec. 4]{Gao}, the authors mentioned that their algorithm may be applied to compute a border basis for an ideal of points. Below we present this algorithm in full detail and a slight improvement.  We point out that in our presentation of this algorithm, 
we use the border term of a polynomial instead of using its leading term. 
Indeed, in view of the structure of the algorithm, we can associate inductively a 
(degree-compatible) border term to each constructed polynomial 
(like the term marking strategy defined in~\cite{kaspar}). In the next algorithm, all newly constructed polynomials are monic, that is the coefficient of the border term of each polynomial is~1. 
First we describe algorithm {\sc BorderTermDivision} to compute the remainder of the 
division of certain polynomials by a given quasi border prebasis.
 
\begin{algorithm}[H]
\caption{{\sc BorderTermDivision}}
\begin{algorithmic}[1]
\STATE {\bf{Input:}} A polynomial $ f $ such that  
    $\text{Supp}(f)\subseteq \widehat{\mathcal{O}}$ and a quasi border 
    prebasis $G=\{g_1,\ldots ,g_\nu\}$

\STATE {\bf{Output:}} A polynomial $\tilde{f}$ in $f+\langle g_1,\ldots, 
     g_\nu \rangle_{\kk}$ such that $\text{Supp}(\tilde{f}) \subseteq 
     \langle \widehat{\mathcal{O}} \rangle_{\kk }$
 
 \STATE Write $f=\sum_{i=1}^{\nu} c_{i}b_{i}+\sum_{j=1}^{\mu} \tilde{c}_{j}t_{j}$ with $c_i,\tilde{c}_{j} \in \kk$ 
\STATE $\tilde{f}:=f-\sum_{i=1}^{\nu} c_{i}g_{i} \in \langle \widehat{\mathcal{O}} \rangle _{\kk}$
\STATE {\bf{return}}$(\tilde{f})$
\end{algorithmic}
\end{algorithm}

Now we are ready for the border basis version of the Farr-Gao Algorithm.

\begin{algorithm}[H]
\caption{{\sc FG-Border}}
\begin{algorithmic}[1]
 \STATE {\bf{Input:}} $ \mathbb{X}= \lbrace P_{1}, \ldots,P_{s} \rbrace \subset \kk^{n}$ where $P_{i}=(p_{i1},\ldots,p_{in})\in \kk^n$
 \STATE {\bf{Output:}} A border pair for $\II(\mathbb{X})$
 \STATE $G:=\{1\}$
\FOR { $  k$ \textbf{from} $  1$ \textbf{to} $s$ }
 \STATE Find a smallest degree polynomial $g_i\in G=\{g_1,\ldots ,g_m\}$ with $ g_{i}(P_{k})\neq 0 $
 \FOR { $  j$ \textbf{from} $ 1 $ \textbf{to} $m $ }
 \IF{$g_j(P_k)\ne 0$}
 \STATE $ g_{j}:= g_{j}-g_{j}(P_{k})/g_{i}(P_{k})\cdot g_{i} $
 \ENDIF
 \ENDFOR
 \STATE $\mathcal{O}:=\mathcal{O} \cup \{\BT_{\mathcal{O}}(g_i)\}$
 
 \STATE $A:=\{\ \}$
 \FOR { $j$ \textbf{from} 1 \textbf{to} $n$}
  \IF {$ x_{j}\cdot \BT_{\mathcal{O}}(g_{i}) \notin \BT_{\mathcal{O}}(G \setminus \lbrace g_{i}\rbrace)$}
 \STATE $ h:=${\sc BorderTermDivision}$((x_{j}-p_{kj})\cdot g_{i}, G)$ 
 \STATE $A:=A \cup \{h\}$
 \ENDIF
 \ENDFOR
 \STATE $ G:=G \setminus \lbrace g_{i}\rbrace $
\STATE  $G:=G\cup A$
\ENDFOR
\STATE {\bf{return}}$( \mathcal{O},G)$
\end{algorithmic}
\end{algorithm}

\begin{lemma}\label{le1}
Let $\mathcal{O} \subseteq \T$ be an order ideal. Let $b \in \partial\mathcal{O}$ 
be an element of smallest degree in~$\partial\mathcal{O}$. 
Then the set $\mathcal{O} \cup \{b\}$ is an order ideal. 
\end{lemma}

\begin{proof}
For each $x_i$ dividing $b$, we have to consider two cases: either $b/x_i \in \partial\mathcal{O}$ or $b/x_i \in \mathcal{O}$. In the fist case, since the term $b$ has the smallest degree in $\partial\mathcal{O}$, this yields a contradiction.  In the latter case, since $\mathcal{O}$ is an order ideal, the set $\mathcal{O} \cup \{b\}$ is closed under
forming divisors. 
\end{proof}
 Below we denote by $\langle X \rangle_\kk$ the $\kk$-vector space generated by a set $X$. 

\begin{theorem}
 Algorithm {\sc FG-Border} terminates and outputs a border pair for the vanishing ideal of its input points. 
\end{theorem}
\begin{proof}
  The termination of the algorithm is ensured by the {\bf for}-loops in the algorithm.  We prove the correctness by induction on $s$. For $s=1$, the algorithm returns $G=\{x_1-p_{11},\ldots ,x_n-p_{1n}\}$ which is a border basis for $\II(\{P_1\})$. Now, suppose that $ \lbrace g_{1},\ldots,g_{m}\rbrace $ is a border basis for $\II(\{P_{1},\ldots,P_{s}\})$ and $\mathcal{O}$ is the corresponding order ideal. We show that the {\bf for}-loop computes for $k=s+1$ a border basis for $\II(\{P_{1},\ldots,P_{s+1}\})$. Let $ P_{k+1}=(p_{k+11},\ldots,p_{k+1n})$ and let $g_i$ be a smallest degree polynomial in $G$ with $ g_{i}(P_{k+1}) \neq 0 $. For each $j$ with $g_j(P_k)\ne 0$, we let $g'_j=g_{j}-g_{j}(P_{k+1})/g_{i}(P_{k+1})\cdot g_{i} $ and we collect all these polynomials in a new set $G'$. Furthermore, we set $\mathcal{O}'=\mathcal{O}\cup \{\BT_{\mathcal{O}}(g_i)\}$. By the choice of $g_i$ and by the inductive hypothesis, Lemma \ref{le1} shows that $\mathcal{O}'$ is an order ideal for $\II(\{P_{1},\ldots,P_{s+1}\})$.
By the choice of $i$, we have $ \BT_{\mathcal{O}}(g'_{j})=\BT_{\mathcal{O}}(g_{j})$ for all $j\neq i$. Also, $g_i$ was replaced with $g_{ij}=(x_{j}-a_{j}) g_{i}$ for $j=1,\ldots ,n$ such that $\BT_{\mathcal{O}}(g_{ij})=x_{j}\BT_{\mathcal{O}}(g_i)$. Thus every element of $G'$ is contained in $\langle \mathcal{O}'\rangle_{\kk} \cup \partial\mathcal{O}'$ which shows that $G'$ is an $\mathcal{O}'$-border prebasis. Since the set $\mathcal{O}'$ generates the $\kk$-vector space $\R/\langle G'\rangle$, and since there is a surjective ring homomorphism $\psi : \R/\langle G'\rangle \longrightarrow \R/\II(\{P_{1},\ldots,P_{s+1}\})$, the set $\mathcal{O}'$ generates the $\kk$-vector space $\R/\II(\{P_{1},\ldots,P_{s+1}\})$. Now the fact that $\mathcal{O}'$ has $s+1$ elements implies that $\mathcal{O}'$ is a basis for the $\kk$-vector space $\R/\II(\{P_{1},\ldots,P_{s+1}\})$ and  $G'$ is the border basis for $\II(\{P_{1},\ldots,P_{s+1}\})$ corresponding to $\mathcal{O}'$.
\end{proof}

The behavior of this algorithm is illustrated by the next example.

\begin{example}
Let us compute a border basis for the vanishing ideal of the set of points 
$\{(1,-1), (3,0),(4,1)\}$ in $\mathbb{Q}^2$. Following the above algorithm, 
suppose that $G=\lbrace g_1,g_2, g_3\rbrace$, where $ g_1=x-2y-3$, $g_2=y^{2}+y$, 
and $g_3=xy-y$, is the border basis constructed for $\II(\{(1,-1), (3,0)\}) $. 

\begin{itemize}
\item[$(1)$] Since $ g_{1}(4,1)\neq 0 $, we update $g_2$ by setting 
$g_2=g_2-g_2(4,1)/g_1(4,1)g_1=y^{2}-3y+2x-6$. By repeating the same process 
with $g_3$ and removing $g_1$ from $G$, we get $G=\{y^{2}-3y+2x-6,xy+3x-7y-9\}$.

\item[$(2)$] Let $ h=(x-4)( x-2y-3)=x^2-7x-2xy+8y+12 $. We apply 
the {\sc BorderTermDivision} algorithm to~$h$ 
and obtain $ A= \lbrace x^2-x-6y-6 \rbrace $.

\item[$(3)$] Now let $ h=(y-1)( x-2y-3)=xy-x-2y^2-y+3 $. 
Since $ xy\in \BT_{\mathcal{O}}(G)$ we do not add~$h$ to~$A$.    

\item[$(4)$] Finally, the set $ G=\{y^{2}-3y+2x-6,xy+3x-7y-9,x^2-x-6y-6\} $ 
is a border basis for the vanishing ideal of the given set of points.
\end{itemize}
\end{example}
 
\begin{remark}
If we remove the condition ``smallest degree'' in algorithm {\sc FG-Border}, 
then the output may be not a border basis. However it is always a quasi border basis. 
Because in each iteration we have  $\BT_{\mathcal{O}}(g_i)\in \partial \mathcal{O}$ 
and $\mathcal{O}$ is a quasi order ideal. For each $t \in \mathcal{O}$ 
and $i\in \{1,\ldots,n\}$, the condition $t/x_i \in \mathcal{O}$ 
implies $t \in \partial(\mathcal{O} \setminus \{t\})$, 
and hence $\mathcal{O} \cup \{\BT_{\mathcal{O}}(g_i)\}$ is a quasi order ideal. 
\end{remark}

Based on algorithm {\sc FG-Border}, we now describe a new algorithm 
that incrementally computes all quasi border pairs for an ideal of points. 

\begin{algorithm}[H]
\caption{{\sc FG-AllQuasiOrderIdeals}}
\label{main:alg}
\begin{algorithmic}[1]
 \STATE {\bf{Input:}} $\mathbb{X}=\{P_1,\ldots,P_s \} \subset \kk^n$
 \STATE {\bf{Output:}} The set of all quasi border pairs for $\II(\mathbb{X})$
 
 \STATE $L:=\emptyset,\mathcal{O}:=\emptyset,G:=\emptyset$
 \STATE $f:=1,d:=1$
 \STATE {\sc QuasiOIStep}$(\mathbb{X},\mathcal{O},d,f,G,L)$
\STATE {\bf{return}}$ (L) $

\end{algorithmic}
\end{algorithm}

Here subalgorithm {\sc QuasiOIStep} is given in Algorithm~8. 
In this algorithm, the function {\sc Interchange}$(L,i,j)$ receives 
a list and integers~$i$ and~$j$, and it interchanges the 
$i$-th and $j$-th elements of~$L$.

\begin{algorithm}[ht]
\caption{{\sc QuasiOIStep}}\label{G:Algorithm}
\begin{algorithmic}[1]
 \STATE {\bf{Input:}} $ \mathbb{X}= \lbrace P_{1}, \ldots,P_{s} \rbrace \subset \kk^{n}$, a quasi order ideal $ \mathcal{O} $, an integer $ d $,  a polynomial $ f \in \R$ in quasi $\mathcal{O}$-border prebasis shape such that $f(P_d)\neq 0$,  a set of polynomials $G$ and a list $L$
 \STATE {\bf{Output:}} An updated list $L$
  \STATE $\mathcal{O}:= \mathcal{O} \cup \lbrace \BT_{\mathcal{O}}(f)\rbrace$;
 \FORALL {$g\in \{h\in G \ | \ h(P_d)\neq 0\} $}
 \STATE $G:=G\setminus \{g\}$
 \STATE $ G:= G \cup\lbrace g-(g(P_d)/f(P_d))\cdot f\rbrace $;
 \ENDFOR{\bf each}
 \STATE $ A:=\{\}$
  \FOR { $  j$ \textbf{from} $ 1 $ \textbf{to} $n $ }
\STATE $ h:=(x_{j}-p_{dj})\cdot f $ where $P_d=(p_{d1},\ldots ,p_{dn})$
\IF {$\BT_{\mathcal{O}}(h) \notin \mathcal{O}$ {\bf and} $\BT_{\mathcal{O}}(h)\notin \BT_{\mathcal{O}}(G)$}
\STATE $h:=${\sc BorderTermDivision}$(h,G)$
\STATE $A:=A \cup \{h\} $
\ENDIF
\ENDFOR
\STATE $G:=G \cup A$
  \IF {$|\mathcal{O}|=s$}
 \STATE $ L:=L\cup \{(\mathcal{O},G)\}$ 
 \ENDIF
 \IF {$|\mathcal{O}|<s$}
 \STATE $d:= \vert \mathcal{O} \vert +1$
 \FOR { $g\in G$ }
 \FOR {$i $ \textbf{from} $ d $ \textbf{to} $ s $}
  \IF {$ g(P_{i})\neq 0 $}
\STATE $ \mathbb{X}:= $ {\sc Interchange} $ (\mathbb{X},i,d) $
 \STATE {\sc QuasiOIStep}$(\mathbb{X},\mathcal{O},d,g,G\setminus \{g\},L) $
 \ENDIF
 \ENDFOR
 \ENDFOR
 \ENDIF
\end{algorithmic}
\end{algorithm}

In order to establish the termination and correctness of this algorithm, 
we state and prove two auxiliary results.

\begin{lemma}\label{quasilem2}
Let $\mathbb{X}\subset \kk^n$ be a finite set of points, $\mathcal{O}$ a quasi order ideal for $\II(\mathbb{X})$ and $G$ its quasi $\mathcal{O}$-border basis for $\II(\mathbb{X})$. Furthermore, let $P\in \kk^n$ be a point so that $P\notin \mathbb{X}$. If $g\in G$ is a polynomial in {\it quasi} $\mathcal{O}$-border prebasis shape with $ g(P)\neq 0 $ and $m=\BT_{\mathcal{O}}(g)$ then $\mathcal{O} \cup \{m\}$ is a quasi order ideal for the ideal of points of $\mathbb{X} \cup \{P\}$.
\end{lemma}

\begin{proof}
 Since $m\in \partial \mathcal{O}$, the set $\mathcal{O} \cup \{m\}$ is a quasi order ideal. It suffices to show that $\langle \mathcal{O} \cup \{m\} \rangle_{\kk} \cap \II(\mathbb{X} \cup \{P\})=\{0\}$. By reductio ad absurdum, suppose that there is a non-zero polynomial $f=m-\sum_{u\in \mathcal{O}}\alpha_uu$ in $\II(\mathbb{X} \cup \{P\})$. We are sure that this polynomial is not equal to $g$. Let $g=m-\tail(g)$ where $\tail(g)$ is a linear combination of terms in $\mathcal{O}$. Since $f\ne g$, we have $\sum_{u\in \mathcal{O}}\alpha_uu\ne \tail(g)$. On the other hand, the fact that $f$ and $g$ are zero on $\mathbb{X}$ implies that $\sum_{u\in \mathcal{O}}\alpha_uu- \tail(g)\in \II(\mathbb{X}) $. However, this non-zero polynomial belongs to $\langle \mathcal{O} \rangle_{\kk}$, in contradiction to the assumption that $\mathcal{O}$ is a quasi order ideal for $\II(\mathbb{X})$. 
\end{proof}

\begin{lemma}\label{lem:13}
Let $\mathbb{X}=\{ P_{1},\ldots,P_{s}\}$ and let $ \mathcal{O} $ be a quasi order ideal for $ \II(\mathbb{X})$. Further, let $ i \in \lbrace 1,\ldots, s-1 \rbrace $, let $ \mathbb{Y} \subseteq \mathbb{X}$ be a subset such that $\# \mathbb{Y}=i$, and let $\mathcal{O}_i \subseteq \mathcal{O} $ is a quasi order ideal for $ \mathbb{Y}$ so that $|\mathcal{O}_i|=i$. Then, for every $ m \in \partial \mathcal{O}_i \cap \mathcal{O} $, there exists a point $ P \in \mathbb{X} \setminus \mathbb{Y} $ such that $\mathcal{O}_i \cup \{m\} $ is a quasi order ideal for $ \II(\mathbb{Y} \cup \{P\}) $. 
\end{lemma}

\begin{proof}
Without loss of generality, suppose that $ \mathbb{Y}=\{ P_{1},\ldots,P_{i}\} $. Since $ m \in \partial \mathcal{O}_i \cap \mathcal{O}$, there exists a polynomial $m-\sum_{u\in \mathcal{O}_{i}}\alpha_uu \in \II(\mathbb{Y})$  where $\alpha_u \in \kk$. Two cases may occur: If there exists $ j \in \lbrace i+1, \ldots ,s \rbrace $, so that $ g(P_j) \neq 0 $, then Lemma \ref{quasilem2} yields that  $\mathcal{O}_i \cup \{m\}$ is a quasi order ideal for $ \II(\mathbb{Y} \cup \{P_j\}) $. Otherwise, we have $ g(P_j)=0 $ for $ j=i+1,\ldots,s $. However, we also have $ g(P_1)= \cdots =g(P_i)=0 $ and therefore $ g $ represents a linear dependency between the elements of $ \mathcal{O} $. This contradicts the fact that  $ \mathcal{O} $ is a quasi order ideal for $ \II(\mathbb{X}) $. This finishes the proof.
\end{proof}

\begin{theorem}
Algorithm {\sc FG-AllQuasiOrderIdeals} terminates and computes all quasi border pairs 
for a given ideal of points.
\end{theorem}

\begin{proof}
First we show that the instructions can be executed. This means that the procedure is 
well-defined. For this purpose, it is enough to show that all polynomials $h$ in Step~10 
and $g$ in Step~22 of Algorithm~8 have an $\mathcal{O}$- border term, i.e., that they are in 
quasi $\mathcal{O}$-border prebasis shape w.r.t.\ the current quasi order ideal~$\mathcal{O}$. 
At first we have $f=1$ and $\BT_{\mathcal{O}}(f)=1$ with respect to $\mathcal{O}=\{\}$. 
Thus the claim is obviously true. In Step~22 of Algorithm~8 we choose $g \in G$ 
and use~$g$ as the new input polynomial~$f$. For the next iteration of Alögorithm~8,  
we add new elements to the set~$G$ in Steps~6 and~16. In the first case we have 
$\BT_{\mathcal{O}}(g-(g(P_d)/f(P_d))\cdot f)=\BT_{\mathcal{O}}(g)$ for every 
$g\in \{h\in G \ | \ h(P_d)\neq 0\} $ and the input polynomial~$f$, 
because we use $\mathcal{O}:= \mathcal{O} \cup \{ \BT_{\mathcal{O}}(f)\}$ 
in Step~3. On the other hand, every $h \in A$ in Step~13 comes from the 
{\sc BorderTermDivision} algorithm. Thus we can conclude that all elements 
of~$G$ are in quasi $\mathcal{O}$-border prebasis shape. 

Next we show that, for every polynomial~$f$ which we use in Algorithm~8, 
we have $f(P_d)\neq 0$. This is true because we have this property for 
the polynomial~$f$ at the beginning (Step 4 of Algorithm~7). 
Also, every time we apply Algorithm~8 recursively, we only apply it to  
a polynomial~$g$ such that $g(P_d)\neq 0$ in Step 24 of Algorithm~8. 
This polynomial~$g$ will be the polynomial~$f$ in the next iteration.

The termination of the algorithm is guaranteed by the fact that $\II(\mathbb{X})$ 
is zero-dimensional. More precisely, if we visualize the computation like a 
tree graph, then at each node the number of branches is finite, 
namely the  cardinality of~$G$ (using the notations of Algorithm~8). 
Moreover, the number of nodes in each branch is bounded by~$s$. 
This implies the termination of the algorithm. 

To prove the correctness of the algorithm, we note that by Lemma \ref{quasilem2} 
and the structure of the algorithm, for each pair $(\mathcal{O},G)$ in~$L$, 
the set $\mathcal{O}$ is a quasi order ideal and~$G$ is the  quasi 
$\mathcal{O}$-border basis for $\II(\mathbb{X})$.
Hence every pair $(\mathcal{O},G)$ in the output is a quasi border pair for $\II(\mathbb{X})$.
  
Now, conversely, we show that every quasi border pair $(\mathcal{O},G)$ 
for $\II(\mathbb{X})$ will be found in~$L$. We proceed by induction on $s$ to show that the  pair $(\mathcal{O},G) $ appears in~$L$. For $s=1$, the assertion is clear. 
Now suppose the assertion holds for~$s$. Let $\mathbb{X}=\{P_1,\ldots ,P_{s+1}\}$, 
and let $(\mathcal{O},G)$ be a quasi border pair for $\II(\mathbb{X})$. 
Let $m$ be a term of maximal degree term in~$\mathcal{O}~$, and let 
$\mathcal{O}' = \mathcal{O} \setminus \{m\}$. 
By Lemma~\ref{lem:13}, there exists the set $\mathbb{Y} \subseteq \mathbb{X} $ 
such that $\mathcal{O}'$ is a quasi order ideal for $\II(\mathbb{Y})$. 
Let $G'$ be the corresponding quasi border basis. Then, by the induction hypothesis, 
the algorithm finds $ (\mathcal{O}',G') $. Let $\mathbb{X} \setminus \mathbb{Y}=\{P\} $. 
Since $ m \in \partial \mathcal{O}' $, there exists the polynomial 
$g'= m-\sum_{u\in \mathcal{O}'}\alpha_uu \in G' $ with $ \alpha _{u} \in \kk $. We note that 
$ g'(P)\neq 0 $, since otherwise $g'$ represents a linear dependency  
between the elements of $\mathcal{O}$ which yields a contradiction. 
Since $g'\in G'$, it is selected in the {\bf for}-loop in Algorithm~8, 
and so~$m$ is added to~$\mathcal{O}'$. This proves the correctness of the algorithm. 
\end{proof}

Let us illustrate the performance of Algorithm~7 by a simple example.

\begin{example}
Let~$\mathbb{X}$ be the set of points $\mathbb{X} = \{(2,3), (5,6), (1,2)\}$
in~$\mathbb{Q}^2$, and let us compute a quasi border basis for $\II(\mathbb{X})$. 
We begin with the pair $(\mathcal{O},G)$, where $\mathcal{O} = \{1,y\}$ and  
$G = \{g_1,g_2,g_3\}$ with $g_1=xy-8y+18$, $g_2=y^2-9y+18$, and $g_3=x+1-y$.
This is a border pair for the ideal of points of $\{\{ (2,3),(5,6)\}$.  
Since $ g_{1}(1,2)\neq 0 $,  we set $\mathcal{O}' = 
\mathcal{O}\cup \{\BT_{\mathcal{O}}(g_{1})\}=\{1,y,xy\}$. 
Therefore $\mathcal{O}'$ is a quasi order ideal 
for $\text{I}(\{ (2,3),(5,6), (1,2) \})$, and the corresponding 
quasi border basis is $ G'=\{ x^2y-9xy+26y-36, xy^2-10xy+26y-36, x+1-y, y^2-xy-y \}$.
Finally, the set of all quasi order ideals for $\mathbb{X}$ is equal to $\{ \{1,x,x^2\}, 
\{1,x,xy\}, \{1,y,y^2\}, \{1,y,xy\} \}$. 
\end{example}

Note that, by using algorithm {\sc FG-AllQuasiOrderIdeals}, we find~1669 
different quasi order ideals for the ideal
$$ 
\II(\lbrace(0,0,0,1), (1,0,0,2), (3,0,0,2),(5,0,0,3),(-1,0,0,4),(4,4,4,5),(0,0,7,6)\rbrace). 
$$
Among them, only 55 sets are order ideals.

\begin{remark}
If we replace "{\bf for} $t$ {\bf in} $S$ {\bf do}" by "{\bf for} $t$ {\bf in} $ \partial \mathcal{O}$ {\bf do}" in algorithm {\sc AllOIStep},
we obtain all quasi border pairs of the input ideal. We call this new algorithm 
{\sc BM-AllQuasiOrderIdeals}, and in the next section, we compare it to {\sc FG-AllQuasiOrderIdeals}. 
\end{remark}

%
%

\section{Experimental Results}\label{timing}

Both Algorithms~\ref{Allordermain} and~\ref{main:alg} have been implemented 
by us in {\sc Maple 2015}. In this section we discuss the efficiency of these 
implementations via a set of benchmarks. For our tests, we consider different 
kinds of sets of points, e.g. complete intersections,  generic sets of points and points on a rational space curve. The results are shown in the following tables where the time and memory columns give, respectively, the consumed CPU time in seconds and the amount of megabytes of memory used by the corresponding algorithm. The last two columns represent, respectively, 
the number of branches and (quasi) order ideals computed by the corresponding algorithm.
All experiments were run on a machine with a 2.40 GHz Intel(R) Core(TM) i7-5500U CPU
and 8~GB of memory.

In Table~\ref{time-BM}, we summarize the results of 
running lgorithm {\sc BM-AllOrderIdeals} on different sets of points.

\setlength{\tabcolsep}{15pt}
\begin{table}[H]
\centering
\caption{ Computing all border pairs}\label{time-BM}
\begin{adjustbox}{max width=\textwidth}
{\footnotesize  
\setlength{\extrarowheight}{3.5pt}
\begin{tabular}{|c||c|c|c|c|c|}
\cline{1-6}
$ \mathbb{X} $  & $\mathbb{A}^n_{\kk_{\mathstrut}}$  & time  &  memory  & $\# $branches & $ \# $order ideals  \\  \cline{1-6} 

\noalign{\vskip-10pt}  & & & & & \\
\cline{1-6}

5 random & $\mathbb{F}_{32003_{\mathstrut}}^{4}$ & 4.33 & 431.98 & 412 & 59  \\ \cline{1-6}
7 random & $\mathbb{F}_{32003_{\mathstrut}}^{2}$ & 1.46 & 168.51 & 230 & 13   \\ \cline{1-6}
8 on twisted cubic & $\mathbb{F}_{32003_{\mathstrut}}^{3}$ &  59.07 & 4365.14 & 2370 & 38  \\ \cline{1-6}
8 complete int. & $\mathbb{F}_{2_{\mathstrut}}^{3}$ & 1.69 & 152.78 & 48 & 1    \\ \cline{1-6}
9 complete int.\ & $\mathbb{F}_{11_{\mathstrut}}^{2}$ & 0.81 & 108.33 & 42 & 1    \\ \cline{1-6}
9 random & $\mathbb{F}_{32003_{\mathstrut}}^{2}$ &  46.81 & 2876.51 & 2618 & 28    \\ \cline{1-6}

\end{tabular}
}
\end{adjustbox}
\end{table}

\setlength{\tabcolsep}{15pt}
\begin{table}[H]
\caption{ Computing all quasi border pairs} \label{tab1}
 \begin{adjustbox}{max width=\textwidth}
 \setlength{\extrarowheight}{2pt}
\begin{tabular}{|c||c|c|l|l|}
\hline
4 random points in $\mathbb{F}_{32003}^{2} $ & time &  memory &  $ \# $branches  & $ \# $quasi order ideals   \\ \hline
{\sc FG-AllQuasiOrderIdeals} & 0.98  & 5.39 & \multirow{2}{*}{\quad \quad 22}  & \multirow{2}{*}{\quad \quad \quad \quad 13} \\  \cline{1-3} 
{\sc BM-AllQuasiOrderIdeals}  & 0.06 & 6.68 & & \\ \hline
  \multicolumn{1}{c}{}  \\ \cline{1-5}
6 random points in $\mathbb{F}_{32003}^{2} $ & time &  memory &  $ \# $branches  & $ \# $quasi order ideals   \\ \hline
{\sc FG-AllQuasiOrderIdeals} & 22.12  & 184.25 & \multirow{2}{*}{\quad \quad  478}  & \multirow{2}{*}{\quad \quad \quad \quad 96} \\  \cline{1-3} 
{\sc BM-AllQuasiOrderIdeals}  &2.58 & 245.66 & & \\ \hline
  \multicolumn{1}{c}{}  \\ \cline{1-5} 
  
type (2,3) complete int. in $\mathbb{F}^{2}_{7} $ & time &  memory &  $ \# $branches  & $ \# $quasi order ideals   \\ \hline
{\sc FG-AllQuasiOrderIdeals} & 0.11  & 7.21 & \multirow{2}{*}{\quad \quad  35}  & \multirow{2}{*}{\quad \quad \quad \quad 4} \\  \cline{1-3} 
{\sc BM-AllQuasiOrderIdeals}  & 0.20 & 23.85 & & \\ \hline

  \multicolumn{1}{c}{}  \\ \cline{1-5} 
  
type (2,2,2) complete int. in $\mathbb{F}^{3}_{2} $ & time &  memory &  $ \# $branches  & $ \# $quasi order ideals   \\ \hline
{\sc FG-AllQuasiOrderIdeals} &   5.08  & 395.12 & \multirow{2}{*}{\quad \,  1020}  & \multirow{2}{*}{\quad \quad \quad \quad 1} \\  \cline{1-3} 
{\sc BM-AllQuasiOrderIdeals}  & 24.04 & 1726.10 & & \\ \hline

  \multicolumn{1}{c}{}  \\ \cline{1-5}
type (3,3) complete int.\ in $\mathbb{F}^{2}_{11} $ & time &  memory &  $ \# $branches  & $ \# $quasi order ideals   \\ \hline
{\sc FG-AllQuasiOrderIdeals} &  17.48  & 615.72 & \multirow{2}{*}{\quad \quad 2368 }  & \multirow{2}{*}{\quad \quad \quad \quad 13} \\  \cline{1-3} 
{\sc BM-AllQuasiOrderIdeals} & 55.27 &2555.12 & & \\ \hline
  \multicolumn{1}{c}{}  \\ \cline{1-5} 
type (3,3) complete int.\ in $\mathbb{F}^{2}_{11} $ & time &  memory &  $ \# $branches   & $ \# $quasi order ideals   \\ \hline
{\sc FG-AllQuasiOrderIdeals} &   114.61  & 1935.25 & \multirow{2}{*}{\quad \quad 3768	 }  & \multirow{2}{*}{\quad \quad \quad \quad 45 } \\  \cline{1-3} 
{\sc BM-AllQuasiOrderIdeals}  &  107.12  & 4567.55136 & & \\ \hline

\end{tabular}
\end{adjustbox}
\end{table}

In the above tables, for example if we look at the first row of Table~2, 
we compute 22 branches to calculate all quasi order ideals for the vanishing ideal 
of the given set of points. But among them there are some repeated results. 
After removing them, we find only 13 different quasi order ideals. 
Moreover, the last two examples in Table~2 show an interesting behavior of quasi
order ideals: for the complete intersection $\langle x(x-1)(x-3),\, 
y(y-1)(y-2)\rangle$ and the complete intersection $\langle x(x-2)(x-7),\,
(y-1)(y-3)(y-5)\rangle$ in~$\mathbb{F}^2_{11}$, there exists one 
order ideal for which they have a border basis, namely 
$\{1,x,y,x^2,xy,y^2,xy^2,y^2x,x^2y^2\}$, but widely different numbers of 
quasi order ideals.

\section*{Acknowledgements.}    The research of the second author 
was in part supported by a grant from IPM (No. 94550420).

\bibliographystyle{plain}
\bibliography{CompAllBB}
\bigbreak

\end{document}